\documentclass[reqno]{amsart}
\usepackage[scale=0.75, centering, headheight=14pt]{geometry}
\usepackage[initials,nobysame]{amsrefs}
\usepackage[T1]{fontenc}
\usepackage{lmodern}
\usepackage[english]{babel}
\usepackage{esint}
\usepackage{mathtools}

\usepackage{xcolor,amsmath,mathrsfs,amssymb,accents,graphicx,epstopdf, amsthm}
\usepackage{ bbold }
\usepackage{comment}
\usepackage{float}
\usepackage{subfig}
\usepackage{tikz}
\usepackage{tikz}
\usepackage{bbm}
\usepackage{amsmath,amssymb,amsfonts,amsthm}
\usepackage{mathtools,accents}
\usepackage{mathrsfs}
\usepackage{xfrac}
\usepackage{array} 
\usepackage{aliascnt}
\usepackage{booktabs} % for much better looking tables
\usepackage{array} % for better arrays (eg matrices) in maths

\usepackage{verbatim} % adds environment for commenting out blocks of text & for better verbatim
\usepackage{subfig} % make it possible to include more than one captioned figure/table in a single float

\usepackage{mathrsfs, dsfont}
\usepackage{amssymb}
\usepackage{amsthm}
\usepackage{amsmath,amsfonts,amssymb,esint}
\usepackage{graphics,color}
\usepackage{enumerate}
\usepackage{mathtools,centernot}

\usepackage{microtype}
\usepackage{paralist} % very flexible & customisable lists (eg. enumerate/itemize, etc.)
\usepackage{cases}
\allowdisplaybreaks

\usepackage{braket}
\usepackage{bm}
\usepackage{xcolor}
\usepackage[colorlinks]{hyperref}
\usepackage{cleveref}
\addto\extrasenglish{}

\usepackage{aliascnt}

\numberwithin{equation}{section}

\newtheorem{question}{Question}[section]
\newtheorem{theorem}{Theorem}[section]
\newtheorem{lemma}[theorem]{Lemma}
\newtheorem{definition}[theorem]{Definition}
\newtheorem{remark}[theorem]{Remark}
\newtheorem{proposition}[theorem]{Proposition}
\newtheorem{corollary}[theorem]{Corollary}

\newcommand{\N}{\mathbb{N}}
\newcommand{\R}{\mathbb{R}}

\newcommand{\Z}{\mathbb{Z}}
\newcommand{\T}{\mathbb{T}}
\newcommand{\Leb}[1]{{\mathscr L}^{#1}} % Misura di Lebesgue

\newcommand{\x}{\times}

\renewcommand{\a}{\alpha}

\newcommand{\e}{\varepsilon}

\newcommand{\s}{\sigma}

\renewcommand{\t}{\tau}

\renewcommand\div{\operatorname{div}}

\newcommand{\supp}{\operatorname{supp}}

\newcommand{\E}{{\mathbb E}}

\newcommand{\vo}{\vec{o}\@ifnextchar{^}{\,}{}}

 \newcommand{\uu}{{\mbox{\boldmath$u$}}}

 \newcommand{\tauV}{{\kern-3pt\tau}}

 \newcommand{\oVVVk}{\overline{\mbox{\boldmath$V$}}\kern-3pt}
 \newcommand{\tVVVk}{\tilde{\mbox{\boldmath$V$}}\kern-3pt}

	\title{On the zero-noise limit for SDE's singular at the initial time} 

\author [J. Pitcho]{Jules Pitcho}
\address{Jules Pitcho
	\hfill\break  GSSI, Via Michele Iacobucci, 2, 67100 L'Aquila, Italy \& ENS  de Lyon, UMPA, 46 all\'ee d'Italie, 
	69364 Lyon,
	France}
\email{jules.pitcho@gssi.it}

\begin{document}
\maketitle
\begin{abstract}
	We investigate the zero-noise limit for SDE's driven by Brownian motion with a divergence-free drift singular at the initial time and prove that a unique probability measure concentrated on the integral curves of the drift is selected. 
	More precisely, we prove uniqueness of the zero-noise limit for divergence-free drifts in $L^1_{loc}((0,T];BV(\T^d;\R^d))\cap L^q((0,T);L^p(\T^d;\R^d))$ where $p$ and $q$ satisfy a Prodi-Serrin condition. 
	The vector field constructed by Depauw \cite{Depauw} lies in this class  and we show that for almost every intial datum, the zero-noise limit selects a probability measure concentrated on several distinct integral curves of this vector field. 
\end{abstract}
\section{Introduction} 
This work studies selection  by a zero-noise limit of probability measures concentrated on the integral curves of a rough vector field. Let us begin by defining integral curves. 
\begin{definition}
	Consider a Borel vector field $b:[0,T]\times\T^d\to\R^d$. We shall say that a continuous curve  $\gamma_x:[0,T]\to\T^d$ is an integral curve of $b$ starting from $x\in\T^d$, if for every $t\in[0,T]$ we have
	\begin{equation}
	\gamma_x(t)-x=\int_0^tb(s,\gamma_x(s))ds \qquad\text{and}\qquad \int_0^T|b(s,\gamma_x(s))|ds<+\infty. 
	\end{equation}
\end{definition}
By the Cauchy-Lipschitz theory, existence and uniqueness of integral curves of $b$ holds when the time integral of the spatial Lipschitz constant of $b$ is finite. However, uniqueness of integral curves can fail almost everywhere when the vector field is rough. Such vector fields are constructed in \cite{pitcho2021nonuniqueness,Depauw,ambrosiocrippaedi,kumar24}. An interesting question is then to select a probability measure concentrated on the set of integral curves starting from $x$. 

Let us review some known selection results. The work of Ambrosio \cite{AmbBV}, building upon the work of DiPerna and Lions \cite{DPL89}, proves that for divergence-free vector fields $b\in L^1((0,T);BV(\T^d;\R^d))$, there exists an essentially unique\footnote{ essential uniqueness is understood in the following sense: for any other Borel family $\{\tilde\gamma_x\}$ satisfying \eqref{eqn_incomp}, it holds $\gamma_x=\tilde\gamma_x$ for Lebesgue almost every $x\in\T^d$.} Borel family of integral curves of $b$ $\{\gamma_x\}$\footnote{ by convention, when the indexing set of a family is not specified, we shall take it to be $\T^d$.} such that for every $t\in [0,T]$, we have 
\begin{equation}\label{eqn_incomp}
\int_{\T^d}\phi(x)dx=\int_{\T^d} \phi(\gamma_x(t))dx\qquad\forall \phi \in C(\T^d) . 
\end{equation}

In the one-dimensional setting, Bafico and Baldi \cite{baficobaldi82} show that for a class of autonomous vector fields, smooth except at a single point, the zero-noise limit selects a probability measure concentrated on two integral curves of $b$ with suitable weights. Delarue and Flandoli \cite{DelarueFlandoli14} also give another proof of the same result.

 In a series two works Bressan, Mazzola and Nguyen \cite{BressanMazzolaNguyen-Markovian23,BressanMazzolaNguyen-Diffusion23} study the one-dimensional problem for autonomous vector fields satisfying a suitable condition, and characterise the probability measures on integral curves compatible with a semigroup structure deterministic or Markovian. Furthermore, these authors show that the transition kernels of these semigroups are limits of the transition kernels of diffusion processes with smooth coefficients. 

We here investigate the zero-noise limit for a subclass of divergence-free vector fields $b\in L^1_{loc}((0,T];BV(\T^d;\R^d))$ including the vector field constructed by Depauw \cite{Depauw}, a two-dimensional vector field for which integral curves are almost everywhere non-unique. Importantly, and in contrast to \cite{baficobaldi82,DelarueFlandoli14}, upon taking the zero-noise limit, we will retain randomness in the initial data, which will be essential to the uniqueness of the zero-noise limit.  
 
 Previous work motivates this investigation. Building upon my work \cite{PitchoArmk23}, in collaboration with Mescolini and Sorella \cite{MescoliniPitcho25}, we show that for vector fields in $b\in L^1_{loc}((0,T];BV(\T^d;\R^d))\cap L^2((0,T)\times\T^d;\R^d)$, the vanishing diffusivity scheme and regularisation by convolution of the vector field select a unique solution of the continuity equation. In \cite{Pitcho24Int}, I also show that for bounded, divergence-free vector fields in $b\in L^1_{loc}((0,T];BV(\T^d;\R^d))$, there exists a unique probabilistic flow of $b$ under which the Lebesgue measure is invariant. For the vector field $b_{DP}$ constructed by Depauw \cite{Depauw}, which is in this class, for almost every initial condition, this flow is concentrated on several distinct integral curves of $b_{DP}$. Moreover, this flow is the weak limit of the classical flows of approximations by convolution of $b$. To introduce our result precisely, let us first fix notations and conventions. 
 
 \subsection{Notations and conventions}
% \begin{itemize} [$\circ$]
 	 	 $\T^d\cong \R^d/\Z^d$ is the $d$-dimensional flat torus. $\Gamma_T$ is the space of continuous paths $C([0,T];\T^d)$. $e_t:\Gamma_T\ni \gamma\longmapsto \gamma(t)\in\T^d$ is the evaluation map at time $t\in[0,T]$. $\mathcal{P}(\Gamma_T)$ is the space of Radon probability measures on $\Gamma_T$. $\Leb{d}$ is the $d$-dimensional Lebesgue measure. $\delta_\gamma$ is the Dirac mass on $\gamma\in\Gamma_T$. $\E(Y)$ is the expactation of the random variable $Y$. $P_Y$ is the law of the random variable $Y$. The Borel $\s$-algebra of a topological space $X$ is denoted $\mathcal{B}(X)$. A measure on a metric space is always assumed to be a Radon measure.
% \end{itemize}

\subsection{Review of SDE theory}
Consider the stochastic differential equation
\begin{equation}\label{eqn_SDE}\tag{SDE}
X^{\nu}_t(x)=x+\int_0^tb(s,X^{\nu}_s(x))ds+\nu W_t,\qquad t\in[0,T]
\end{equation}
on $\T^d$, where $b:[0,T]\times\T^d\to \R^d$ is a Borel vector field, $\nu> 0$, and $(W_t)_{t\in[0,T]}$ is a $\T^d$-valued Brownian motion started from zero on the canonical filtered probability space $(\Omega,\mathcal{F}, (\mathcal{F}_t)_{t\in[0,T]},P)$. Even when integral curves of $b$ starting from $x$ are non-unique, strong existence and strong uniqueness of solutions of \eqref{eqn_SDE} holds. 
This phenomenon is called regularisation by noise in the litterature. We say that  continuous adapted process $(X^\nu_t(x))_{t\in[0,T]}$ on the filtered probability space $(\Omega,\mathcal{F}, (\mathcal{F}_t)_{t\in[0,T]},P)$ is a strong solution, if $P$-almost surely, we have \eqref{eqn_SDE}. We say that strong uniqueness holds for \eqref{eqn_SDE}, if for any two strong solutions are indistinguishable. 
% $(X^\nu_t(x))_{t\in[0,T]}$ and $(\tilde X^\nu_t(x))_{t\in[0,T]}$, $P$-almost surely, we have $\sup_{t\in[0,T]}|X_t-\tilde X_t|=0.$ 

\bigskip 
Veretennikov \cite{Veretennikov82} proves strong existence and strong uniqueness for \eqref{eqn_SDE} when $b$ is bounded. Krylov and R\"ockner \cite{KrylovRockner05} prove strong existence and strong uniqueness for \eqref{eqn_SDE} when $b$ belongs to $L^q_p(T)$ defined below in \eqref{eqn_function_space} where $p$ and $q$ satisfy \eqref{eqn_prodi_serrin}. Fedrizzi and Flandoli \cite{FlandoliFedrizzi11} then give another proof of this result and further show that the random continuous sample paths $\{ (X^{\nu}_t (x))_{t\in[0,T]}\;:\;x\in\T^d\}$ have a continuous modification in $x$.

 In order to include randomness in the initial data, let us form the filtered probability space 
\begin{equation} \label{eqn_proba_space}
\mathscr{P}:=(\Omega\times \T^d, \mathcal{F}\otimes \mathcal{B}(\T^d),(\mathcal{F}_t\otimes \mathcal{B}(\T^d))_{t\in[0,T]},P\otimes\Leb{d}).
\end{equation} 
We also define the function space 
\begin{equation}\label{eqn_function_space}
L^q_p(T):=L^q((0,T);L^p(\T^d;\R^d)),
\end{equation}
which forms a Banach space under the norm $$\|b\|_{L^q_p}:=\Big(\int_0^T\big(\int_{\T^d}|b(t,x)|^pdx\big)^{q/p} dt\Big)^{1/q}. $$
We assume that the exponents $p$ and $q$ satisfy the following Prodi-Serrin condition 
\begin{equation}\label{eqn_prodi_serrin} 
\frac{d}{p}+\frac{2}{q}<1, \qquad p,q\in (1,+\infty].
\end{equation}
We then view $X^\nu=(X_t^{\nu})_{t\in[0,T]}$ as a random continuous process on $\mathscr{P}$, with randomness coming both from the Brownian motion and from the initial data. 
 %$$\E(Y):=\int_{\Omega\times \T^d}Y d (P\otimes \Leb{d})$$ will denote the expectation of random variables this probability space. More precisely 
For future use, we record the following proposition extracted from \cite{FlandoliFedrizzi11}. 
\begin{proposition}\label{prop_noisy_flow} 
	Consider a Borel vector field $b:[0,T]\times\T^d\to \R^d$, and $p,q\in (1,+\infty]$ satisfying \eqref{eqn_prodi_serrin}. 
	Assume that $b\in L^q_p(T)$. Then there exists a continuous stochastic process $X^\nu$ on $\mathscr{P}$ defined in \eqref{eqn_proba_space} such that
	\begin{enumerate} 
		\item	 $P$-almost surely, we have 
		\begin{equation}
		X_t^{\nu}(x)=x+\int_0^tb(s,X^{\nu}_s(x))ds+\nu W_t \quad\text{and}\quad \int_0^T|b(s,X^{\nu}_s(x))|^2ds<+\infty\qquad\forall x\in\T^d;
		\end{equation}
		\item $P$-almost surely, the random family of continuous paths $\{(X^\nu_t(x))_{t\in[0,T]} : x\in\T^d\}$ depends continuously on $x$. 
	\end{enumerate} 
	Any two processes satisfying $(i)$ and $(ii)$ are indistinguishable. 
	\bigskip 
	
	If we further assume that $b$ is divergence-free, then we have
	\begin{equation}
	\E\phi(X_t^{\nu})=\int_{\T^d}\phi(x)dx\qquad\forall\phi\in L^1(\T^d)\;\text{ and }\;\forall t\in[0,T]. 
	\end{equation}
\end{proposition}

% It then becomes interesting to study the vanishing noise limit in \eqref{eqn_SDE} and to determine whether a probability measure over integral curves is selected.
%  In one space dimension, a first result was given by Bafico and Baldi for vector fields smooth except at a single point. 

%There exists a unique solution strong solution in the class of continuous stochastic processes such that 
%$$P\Big(\int_0^T|b(s,X^{x,\nu}_s)|^2ds<+\infty\Big)=1.$$
%Denote by $X^{\nu}$ the strong solution of \eqref{eqn_SDE}. An important fact is that the random continuous sample paths $\{ (X^{x,\nu}_t )_{t\in[0,T]}\;:\;x\in\T^d\}$ have a continuous modification in $x\in\T^d$, namely almost surely we have the map $\T^d\ni x\longmapsto (X^{x,\nu}_t)_{t\in[0,T]}\in C([0,T];\T^d)$ is continuous. 

%A consequence of this fact is that, for this continuous modification, if $b$ is weakly divergence-free, then also surely, we have 
%\begin{equation}
%\int_{\T^d}\phi(X_t^{x,\nu}(\omega))dx=\int_{\T^d}\phi(x)dx\qquad\forall \phi \in C(\T^d).
%\end{equation}

%In this context, it is natural to consider probability measures over the set of solutions of integral curves of $b$. 
Under the hypothesis of \Cref{prop_noisy_flow}, we can study the zero-noise limit for \eqref{eqn_SDE}, whilst retaining randomness in the initial data. 
This leads to the following definition. 
\begin{definition}
	Consider a Borel vector field $b:[0,T]\times \T^d\to \R^d$ in $L^q_p(T)$ and assume that \eqref{eqn_prodi_serrin} holds. We shall say that a probability measure $\eta$ in $\mathcal{P}(\Gamma_T)$ is a zero-noise flow of $b$, if there exists a sequence $(\nu_n)_{n\in\N}$ of real numbers in $(0,1)$ such that $\nu_n \downarrow 0$ and the laws  $P_{X^{\nu_n}}$ converge to $\eta$ in $\mathcal{P}(\Gamma_T)$ as $n\to+\infty$. 
\end{definition}
The following question is then natural. 
\begin{question}
Consider a Borel vector field $b:[0,T]\times\T^d\to \R^d$ in $L^q_p(T)$ and assume that \eqref{eqn_prodi_serrin} holds.	What are the minimal assumptions on $b$ such that there exists a unique zero-noise flow of $b$? 
\end{question}
We will prove in \Cref{prop_vanishing_noise} that a zero-noise flow of $b$ is concentrated on integral curves of $b$ and is incompressible if $b$ is divergence-free. Ambrosio's work \cite{AmbBV} thus implies that for the class of divergence-free vector fields in $L^1((0,T);BV(\T^d;\R^d))$, there exists a unique zero-noise flow of $b$ induced by a flow map, which he called the regular Lagrangian flow of $b$.  In \cite{AmbBV}, the uniqueness of the regular Lagrangian flow is deduced from the uniqueness of bounded weak solutions for the initial value problems for continuity equation.
 However for the vector field $b_{DP}$ constructed by Depauw \cite{Depauw}, there are infinitely many bounded weak solution to these initial value problems. There is therefore no reason \emph{a priori} for uniqueness of the zero-noise flow of $b_{DP}$ nor for zero-noise flows of $b_{DP}$ to be induced by a flow map. 
 
 In this work we show for vector fields $b$ in $L^1_{loc}((0,T];BV(\T^d;\R^d))\cap L^q_p(T)$ such that \eqref{eqn_prodi_serrin} holds, there exists a unique zero-noise flow of $b$. This class includes $b_{DP}$ for which the zero-noise flow is concentrated on several
 distinct integral curves of $b_{DP}$ for almost every initial datum.
Let us now present a useful tool for our purpose: the disintegration of a measure with respect to a Borel map and a target measure used in the study of linear transport in \cite{ABC14,BianchiniBonicatto20,Pitcho24Int}.
 %In the study of weak solutions of the continuity equation, disintegration is used
% in \cite{ABC14} by Alberti, Bianchini and Crippa to establish the optimal uniqueness result for the continuity equation
 %along a bounded, divergence-free, and autonomous in the two-dimensional setting. In \cite{BianchiniBonicatto20}, Bianchini and
 %Bonicatto also use disintegration to prove a uniqueness result for nearly incompressible vector fields in
%$L^1_tBV_x$. In \cite{Pitcho24Int}, I use disintegration to prove that for the Depauw vector field $\b_{DP}$, there is a unique incompressible probability measure in $\mathcal{P}(\Gamma_T)$ concentrated on integral curves of $b_{DP}$, which is supported on several distinct integral curves of $b_{DP}$ for almost every initial datum. 
\subsection{Disintegration of a measure} \label{subsec_disintegration}
Let $X$ and $Y$ be a separable metric spaces, $\mu$ a positive measure on $X$, $\nu$ a positive measure on $Y$ and $f:X\to Y$ a Borel map such that $f_\#\mu=\nu$. Assume that $\mu$ is tight. Then there exists a Borel family of probability measures $\{\mu_y:y\in Y\}$ of measures on $Y$ such that 
\begin{enumerate}
	\item $\mu_y$ is concentrated on the level set $E_y:=f^{-1}(y)$ for every $y\in Y$;
	\item the measure $\mu$ can be decomposed as $\mu=\int_Y\mu_yd\nu(y)$ , which means that 
	\begin{equation}
	\mu(A)=\int_Y\mu_y(A)d\nu(y). 
	\end{equation}
\end{enumerate}
Any family satisfying $(i)$ and $(ii)$ is called a \emph{disintegration} of $\mu$ with respect to $f$ and $\nu$. The disintegration is essentially unique in the following sense: for any other disintegration $\{\tilde\mu_y: y\in Y\}$, it holds $\mu_y=\tilde\mu_y$ for $\nu$-a.e. $y\in Y$. 
We also have
\begin{equation}\label{eqn_int_disintegration}
\int_X \phi d\mu=\int_Y\Big[\int_{E_y}\phi d\mu_y\Big]d\nu(y)
\end{equation}
for every $\phi \in L^1(X,\mu)$. 

\bigskip 

We now give a useful fact.
Let $g:X\to X$ and $h:X\to Y$ be Borel maps %such that $(f\circ g)_\#\mmu\ll \nnu,$
such that:
\begin{itemize}
	\item[(P)] $h_\#\mu=\nu$ and for every $y\in Y$, we have $g^{-1}((h^{-1}(y))^c)=(f^{-1}(y))^c$. 
\end{itemize} 
The following is true. 
\begin{lemma} \label{lem_disintegration_push_forward} 
	In the context of this paragraph, if $\{\mu_y : y\in Y\}$ is a disintegration of $\mu$ with respect to $ f$ and $\nu$, then $\{g_\# \mu_y : y\in Y\}$ is a disintegration of $g_\#\mu$ with respect to $ h$ and $\nu$. 
\end{lemma} 
\begin{proof}
	Let $y\in Y$. Observe that $g_\#\mu_y((h^{-1}(y))^c)=\mu_y(g^{-1}(h^{-1}(y))^c)=\mu_y((f^{-1}(y))^c)=0$ where we have used (P) in the second to last equality and that $\mu_y$ is concentrated on $f^{-1}(y)$ in the last equality. 
	%	As $g$ is continuous, we have $\supp g_\#\mmu_y=g(\supp \mmu_y),$ by to \cite[Theorem 1.8]{Mattila}. 
	%	We know that $\supp\mmu_y$ is contained in $f^{-1}(y)$. So with (P), this yields $\supp g_\#\mmu_y=g(\supp\mmu_y)\subset g(f^{-1}(y))\subset g^{-1}f^{-1}(y)$. 
	So	$g_\#\mu_y$ is supported on $h^{-1}(y)$, and since $y$ was arbitrary, this proves $(i)$. 
	
	Let $A$ a Borel set in $X$. Then as $g^{-1}(A)$ is a Borel set in $X$, it follows that 
	\begin{equation}
	g_\#\mu(A)=\mu(g^{-1}(A))=\int_Y \mu_y (g^{-1}(A))d\nu(y)=\int_Y g_\#\mu_y (A)d\nu(y),
	\end{equation}
	which gives $(ii)$. 
\end{proof}
\subsection{Statement of the result} 
Recall that we denote by $b_{DP}:[0,T]\times\T^d\to \R^d$ the vector field constructed by Depauw in \cite{Depauw} (see the \Cref{sec_const_depauw} for a construction). 
The following is the main result of this paper. 

\begin{theorem}\label{thm_main} 
	Consider a divergence-free Borel vector field $b:[0,T]\times\T^d\to \R^d$, and $p,q\in (1,+\infty]$ satisfying \eqref{eqn_prodi_serrin}. Assume that $b\in L^1_{loc}((0,T];BV(\T^d;\R^d))\cap L^q_p(T)$. Then there exists a unique zero-noise flow $\eta$ of $b$. Furthermore, 
	\begin{enumerate} 
		\item for every Borel vector field $a$ such that $a=b$ $\Leb{d+1}$-a.e., $\eta$ is also the unique zero-noise flow of $a$;
	\item 	there exists a Borel family $\{\gamma_{y}\}$ in $\Gamma_T$, a Borel family of probability measures $\{\tilde\nu_x\}$ on $\T^d$ such that for every disintegration $\{\eta_{0,x}\}$ of $\eta$ with respect to $e_0$ and $\Leb{d}$, we have
	\begin{equation}
\eta_{0,x}=	\int_{\T^d}\delta_{\gamma_{y}}d\tilde\nu_x(y) \qquad\text{for $\Leb{d}$-a.e. $x\in\T^d$},
	\end{equation}
	and for $b=b_{DP}$, the probability measures $\tilde\nu_x$ are not Dirac masses for $\Leb{d}$-a.e. $x\in\T^d$. 
	\end{enumerate}
	\begin{comment} 
	Moreover $\eta$ is characterised by the following:
	\begin{enumerate}
\item	 for $P_X$-a.e. $\gamma\in \Gamma_T$, we have 
	\begin{equation}
	\gamma(t)=\gamma(0)+\int_0^tb(s,\gamma(s))ds \qquad\forall t\in[0,T];
	\end{equation}
\item 	for every $\phi \in C(\T^d)$, we have 
	\begin{equation}
	\int_{\Gamma_T}\phi(\gamma(t))P_X(d\gamma)=\int_{\T^d}\phi(x)dx.
	\end{equation}
	\end{enumerate} 
	\end{comment} 
	\end{theorem} 

	It would be interesting to understand whether, in the context of this theorem, the convergence to the zero-noise flow can be strengthened. 
It would also be interesting to construct vector fields for which uniqueness of the zero-noise flow fails. Note that vector fields for which the vanishing diffusivity scheme is not a selection criterion for the continuity equation are constructed in \cite{colombo2022anomalous, huysmans2023nonuniqueness}. Note also that vector fields for which smooth regularisation is not a selection criterion are constructed in \cite{Pitcho23selection, DeLellis_Giri22,colombo2022anomalous,CiampaCrippaSpirito19}. 

\subsection{Our argument and organisation of this paper} 
The existence of a zero-noise flow follows by tightness of the family of laws $\{P_{X^\nu}\}_{\nu>0}$ and we show that it must be concentrated on integral curves of the vector field. For the uniqueness part of our result, we first observe that a zero-noise flow of a divergence-free vector field must satisfy an incompressibility condition. For vector fields which are further assumed to be in $L^1_{loc}((0,T];BV(\T^d;\R^d))$, we then show that there is a unique probability measure concentrated on integral curves of the vector field satisfying this incompressibility condition, whence it must be the zero-noise flow. The proof of the integral representation of point $(ii)$ and that this representation is stochastic in nature for $b_{DP}$ draws from \cite{Pitcho24Int} and utilises the disintegration of a measure.

In Section \ref{sec_uniqueness}, we prove existence and uniqueness of the zero-noise flow for the vector fields adressed by our theorem. In Section \ref{sec_stoch}, we prove the integral representation of point $(ii)$ and show that it is stochastic for $b_{DP}$. 
\subsection*{Acknowledgements}
This work was completed at the Gran Sasso Science Institute, whose hospitality I gratefully acknowledge. I am also thankful to Jan Burczak and L\'aszl\'o Sz\'ekelyhidi Jr. for their encouragements to write this paper. I am thankful to the anonymous reviewers for useful suggestions.
\section{Existence and uniqueness of the zero-noise flow}\label{sec_uniqueness}
\subsection{Existence of a zero-noise flow} The following will serve as our existence result. We also gather useful properties of zero-noise flows for divergence-free vector fields. 
\begin{proposition}\label{prop_vanishing_noise}
	Consider a divergence-free Borel vector field $b:[0,T]\times\T^d\to\R^d$ and $p,q\in (1,+\infty]$ satisfying \eqref{eqn_prodi_serrin}. Assume that $b\in L^q_p(T)$. Then there exists a zero-noise flow $\eta$ of $b$, which further satisfies
	\begin{enumerate} 
\item	$\eta$ is concentrated on integral curves of $b$, i.e. 
\begin{equation}
\int_{\Gamma_T}\Big|\gamma(t)-\gamma(0)-\int_0^tb(s,\gamma(s))ds\Big|\eta(d\gamma) \qquad\forall t\in[0,T];
\end{equation}
\item for every $t\in[0,T]$, we have $(e_t)_\#\eta=\Leb{d}$. 
	\end{enumerate} 
%Furthermore, if there exists a unique $\eta$ in $\mathcal{P}(\Gamma_T)$ such that $(i)$ and $(ii)$ hold, then $\eta$ is the unique vanishing noise flow of $b$. 
\end{proposition}
\begin{remark}\label{rmk_unique_vanishing_noise}
	In view of the above proposition, if there exists a unique $\eta$ satisfying $(i)$ and $(ii)$, then $\eta$ is the unique zero-noise flow of $b$, and a subsubsequence argument shows that the whole family $\{P_{X^\nu}\}_{\nu>0}$ of laws of $X^\nu$ converges to $\eta$ in $\mathcal{P}(\Gamma_T)$ as $\nu\downarrow 0$. 
\end{remark}
Let us introduce the Gagliardo semi-norm. For $u:[0,T]\to \T^d$ Borel, we define
\begin{equation}
[u]_{W^{\a,p}}:=\Big(\int_0^T\int_0^T\frac{|u( s)-u(t)|^p}{|s-t|^{1+\a p}}dsdt\Big)^{1/p}.
\end{equation}
The fractional Sobolev spaces are then defined as
\begin{equation}
W^{\a,p}((0,T;\T^d)):=\Big\{u\in L^p((0,T);\T^d)\;:\; [u]_{W^{\a,p}}<+\infty \Big\},
\end{equation}
and form Banach spaces under the norm $\|\cdot\|_{W^{\a,p}}:=\|\cdot\|_{L^p}+[\cdot]_{W^{\a,p}}.$ 
Let us introduce the H\"older semi-norm
\begin{equation}
[u]_{C^\theta}:=\sup_{s,t\in [0,T], s\neq t} \frac{|u(t)-u(s)|}{|t-s|^\theta}. 
\end{equation}
The H\"older spaces are defined as 
\begin{equation}
C^{\theta}((0,T);\T^d):=\Big\{u\in C^0((0,T);\T^d) \;:\;[u]_{C^\theta}<+\infty\Big\},
\end{equation}
and form Banach spaces under the norm $\|\cdot\|_{C^\theta}:=\|\cdot\|_{C^0}+[\cdot]_{C^\theta}.$
The following continuous embeddings then hold
\begin{equation}\label{eqn_Sobolev_embedding} 
W^{\a,p}((0,T);\T^d)\hookrightarrow C^{\theta}((0,T);\T^d),\qquad \text{for }\;\a \geq \theta +\frac{1}{p}, \qquad\theta\geq 0, \;p\geq 1.
\end{equation}
\begin{proof}
	\textbf{Step 1.} (Compactness)	
	Let us show that family of laws $\{P_{X^\nu}\}_{\nu>0}$ is precompact in $\mathcal{P}(\Gamma_T)$. 
	By Prokhorov's theorem, it is enough to prove tightness of the family $\{P_{X^\nu}\}_{\nu>0}$. 
	By the Markov inequality, for every $R>0$, we have 
	\begin{equation*}
	(P\otimes \Leb{d})\Big([X^{x,\nu}]_{W^{\a,p}}>R\Big)\leq \frac{1}{R}\int_0^T\int_0^T\frac{\E[|X^{\nu}_t-X^{\nu}_s|^p]}{|t-s|^{1+\a p}}dsdt,
	\end{equation*}
	Now, for every $s,t\in[0,T]$, we have
	\begin{equation}
	\begin{split} 
	\E[|X^{\nu}_t-X^{\nu}_s|^p]&\leq C_p\E[\int_s^t|b(s,X^{x,\nu}_s)|ds]^p+C_p\nu^p\E[|W_t-W_s|^p]\\
	&\leq C_p|t-s|^{p/2} \| b\|^p_{L^q((s,t);L^p(\T^d))}+C_p\nu^p|t-s|^{p/2},
	\end{split}
	\end{equation}
	for some constant $C_p$ which depends only on $p$ and where for the first term, we have used the law of $X^{\nu}_s$ is the Lebesgue measure on $\T^d$, that $q>2$ by \eqref{eqn_prodi_serrin}, and the Burkholder-Davis-Gundy inequality for the second term. 
	Therefore we have
	\begin{equation} 
	(P\otimes \Leb{d})\Big([X^{x,\nu}]_{W^{\a,p}}>R\Big)\leq \frac{C_{T,p,b}}{R}\int_0^T\int_0^T|t-s|^{-1-(\a-1/2)p}dsdt,
	\end{equation} 
	where the constant $C_{T,p,b}$ depends only on $T$, $p$ and $ \| b\|_{L^q_p}.$
	This integral is finite for $\a<1/2$. Now fix the parameter $p>2$ which in view of \eqref{eqn_Sobolev_embedding} is consistent with $\theta>0$. 
By the Sobolev embedding \eqref{eqn_Sobolev_embedding} , there exists a constant $C=C_{T,p,\a,\theta,b}>0$ such that for every $R>0$, we have
	$$(P\otimes \Leb{d})\Big([X^{x,\nu}]_{C^\theta}>R+C\Big)\leq \frac{C}{R}\int_0^T\int_0^T|t-s|^{-1-(\a-1/2)p}dsdt.$$
	By Ascoli's theorem, the sets $\Big\{\gamma\in \Gamma_T:[\gamma]_{C^\theta}\leq R+C \Big\}$ are compact in $\Gamma_T$. 
	This implies tightness of the family $\{P_{X^\nu}\}_{\nu>0}$. 

\bigskip 

\textbf{Step 2.} (Incompressibility and concentration on integral curves) Let us prove $(ii)$. 
	Consider a sequence $(P_{X^{\nu_n}})_{n\in\N}$ converging narrowly to a probability measure $P$ in $\mathcal{P}(\Gamma_T)$ as $\nu_n\downarrow 0$, which exists by Step 1. Let $\phi\in C(\T^d)$ and let $t\in[0,T]$. Then the functional $\Gamma_T\ni \gamma\longmapsto \phi(\gamma(t))\in\R$ belongs to $C_b(\Gamma_T)$. Therefore, we have 
	\begin{equation}
\int_{\Gamma_T}\phi(\gamma(t))P_X(d\gamma)=\lim_{\nu_n\downarrow 0}	\int_{\Gamma_T}\phi(\gamma(t))P_{X^{\nu_n}}(d\gamma)=\int_{\T^d}\phi(x)dx. 
	\end{equation}
	As $t$ and $\phi$ were arbitrary, this proves $(ii)$.
	
	\bigskip
	Let us prove $(i)$. It suffices to show that $$\int_{\Gamma_T}\Big|\gamma(t)-\gamma(0)-\int_0^tb(s,\gamma(s))ds\Big|P(d\gamma)=0\qquad \forall t\in[0,T].$$ 
	Let $\e>0$, $c:[0,T]\times\T^d\to\R^d$ be a continuous vector field such that $\int_0^T\int_{\T^d}|c(s,x)-b(s,x)|dxds<\e$ and $t\in[0,T]$. We then have
	\begin{equation}
	\begin{split} 
&	\int_{\Gamma_T}\Big|\gamma(t)-\gamma(0)-\int_0^tb(s,\gamma(s))ds\Big|P(d\gamma)\\
	&\leq \int_{\Gamma_T} \Big|\gamma(t)-\gamma(0)-\int_0^tc(s,\gamma(s))ds\Big| P(d\gamma)+\int_{\Gamma_T} \Big|\int_0^t c(s, \gamma(s))-b(s,\gamma(s))|P(d\gamma)\\
	&\leq \limsup_{n\rightarrow+\infty}\int_{\Gamma_T} \Big|\gamma(t)-\gamma(0)-\int_0^tc(s,\gamma(s))ds\Big| P^{\nu_n}(d\gamma)+\int_{\Gamma_T} \Big|\int_0^t c(s, \gamma(s))-b(s,\gamma(s))\Big|P(d\gamma)\\
	&\leq \limsup_{n\rightarrow+\infty}\E\Big[ \Big|\nu_n W_t+\int_0^t(b(s,X^{\nu_n}_s)-c(s,X^{\nu_n}_s))ds\Big|\Big] + \int_0^t\int_{\T^d}| c(s, x)-b(s,x)|dxds\\
	&\leq\limsup_{n\rightarrow+\infty} \nu_n\E[\sup_{t\in[0,T]}|W_t|]+2\int_{\T^d} \int_0^t| c(s, x)-b(s,x)|dxds<2\e. 
	\end{split} 
	\end{equation}
	
As $\e$ was arbitrary, we have thus shown that any zero-noise flow must satisfy $(i)$ and $(ii)$. Therefore, if there exists a unique probability measure $\eta$ satisfying $(i)$ and $(ii)$, a subsubsequence argument shows that the whole family of laws $\{P_{X^\nu}\}_{\nu>0}$ converges narrowly to $\eta$ as $\nu \downarrow 0$. This proves the thesis. 
\end{proof}

\subsection{Uniqueness of the zero-noise flow} 
We will now prove uniqueness of the zero-noise flow. 
The following theorem can be extracted from Ambrosio \cite{AmbBV}. 
\begin{theorem}\label{thm_ambrosio}
	Consider a divergence-free vector field $b:[0,T]\times\T^d\to \R^d$. Assume that $b\in L^1((0,T);BV(\T^d;\R^d))$. Then there exists a unique probability measure $\eta$ in $\mathcal{P}(\Gamma_T)$ such that
	\begin{enumerate} 
	\item	 $\eta$ is concentrated on integral curves of $b$, i.e.
	\begin{equation}
\int_{\Gamma_T} \Big| \gamma(t)-\gamma(0)-\int_0^t b(s,\gamma(s))ds\Big|\eta(d\gamma)=0 \qquad\forall t\in[0,T];
	\end{equation}
	
	\item for every $t\in[0,T]$, we have $(e_t)_\#\eta=\Leb{d}$.
\end{enumerate}
\bigskip 

Furthermore, there exists a Borel family $\{\gamma_{y}\}$ of $\Gamma_T$ such that for every disintegration $\{\eta_{0,y}\}$ of $\eta$ with respect to $e_0$ and $\Leb{d}$, we have $\eta_{0,y}=\delta_{\gamma_{y}}$ for $\Leb{d}$-a.e. $y\in\T^d$. 
\end{theorem}

In order to relate uniqueness of the zero-noise flow in our setting to \Cref{thm_ambrosio}, let us first introduce $\tau\in (0,T)$ and define the maps
{
\begin{equation} \label{eqn_def_S}
 S^\tau:\Gamma_T\ni \gamma(\cdot)\longmapsto \gamma(\t\vee\cdot)\in\Gamma_T
 \quad\text{ and } \quad B^\tau:\Gamma_T\ni \gamma(\cdot)\longmapsto \gamma(T-\t\wedge\cdot)\in\Gamma_T.
 \end{equation} 
Notice that  (P) of \Cref{subsec_disintegration} is satisfies with $g=S^\tau$ and $f=h=e_T$ or with $g=B^\t$ and $f=h=e_0$. }
 Define also the vector field 
\begin{equation}\label{eqn_def_b^t}
b^\tau(t,x):=\left\{\begin{split} b(t,x)\qquad &\text{if $\t>t$},\\
0\qquad &\text{if $t\leq \t$.}\end{split}\right. 
\end{equation}
We now record the following elementary fact. 
\begin{lemma}\label{lem_weak_conv} 
	Let $\mu$ be a probability measure in $\mathcal{P}(\Gamma_T)$. Then $(S^\tau)_\#\mu$ converges narrowly to $\mu$ as $\tau\downarrow 0$ and $(B^\tau)_\#\mu$ converges narrowly to $\mu$ as $\tau\downarrow 0$.
\end{lemma}
\begin{proof}
	Let $\Phi\in C_b(\Gamma_T)$. For every $\gamma\in \Gamma_T$, we have $\lim_{\tau\downarrow 0}S^\tau\gamma=\gamma$, which by continuity implies $\lim_{\tau\downarrow 0}\Phi(S^\tau\gamma)=\Phi(\gamma)$. We also clearly have 
	\begin{equation}
	\int_{\Gamma_T} |\Phi(S^\tau\gamma)|\mu(d \gamma)\leq \|\Phi\|_{C^0}. 
	\end{equation}
	Therefore, by dominated convergence, it holds 
	\begin{equation}
	\lim_{\tau\downarrow 0}\int_{\Gamma_T}\Phi(\gamma) (S^\t)_\#\mu(d\gamma)=\lim_{\tau\downarrow 0}\int_{\Gamma_T}\Phi(S^\tau\gamma)\mu(d\gamma)=\int_{\Gamma_T}\Phi(\gamma)\mu(d\gamma). 
	\end{equation}
	The same argument shows that $(B^\t)_\#\mu$ converges narrowly to $\mu$ as $\t\downarrow0$. 
\end{proof}
Let us record also the following consequence of Theorem \ref{thm_ambrosio}. 
\begin{corollary}
	Consider a divergence-free vector field $b:[0,T]\times\T^d\to \R^d$. Assume that $b\in L_{loc}^1([0,T);BV(\T^d;\R^d))\cap L^1((0,T)\times\T^d;\R^d)$. Then there exists a unique probability measure $\eta$ in $\mathcal{P}(\Gamma_T)$ such that
	\begin{enumerate} 
		\item	 $\eta$ is concentrated on integral curves of $b$, i.e.
		\begin{equation}
		\int_{\Gamma_T} \Big| \gamma(t)-\gamma(0)-\int_0^t b(s,\gamma(s))ds\Big|\eta(d\gamma)=0 \qquad\forall t\in[0,T];
		\end{equation}
		
		\item for every $t\in[0,T]$, we have $(e_t)_\#\eta=\Leb{d}$.
	\end{enumerate}
	\bigskip 
	
	Furthermore, there exists a Borel family $\{\gamma_{y}\}$ of $\Gamma_T$ such that for every disintegration $\{\eta_{0,y}\}$ of $\eta$ with respect to $e_0$ and $\Leb{d}$, we have $\eta_{0,y}=\delta_{\gamma_{y}}$ for $\Leb{d}$-a.e. $y\in\T^d$. 
\end{corollary}
\begin{proof}
	The existence of a probability measure $\eta$ in $\mathcal{P}(\Gamma_T)$ satisfying $(i)$ and $(ii)$ follows from Smirnov's representation for normal 1-currents \cite{Smirnov94} or Ambrosio's superposition principle \cite{ambrosiocrippaedi}. 
	
	Let us now prove the last part of the statement. Let $\{\eta_{0,y}\}$ be a disintegration of $\eta$ with respect to $e_0$ and $\Leb{d}$.
	By \Cref{lem_disintegration_push_forward}, for every $k\in\N$, the family $\{(B^{1/k})_\#\eta_{0,y}\}$ is a disintegration of $(B^{1/k})_\#\eta$ with respect to $e_0$ and $\Leb{d}$. 
	One can check directly that $(B^{1/k})_\#\eta$ satisfies $(i)$ and $(ii)$ of Theorem \ref{thm_ambrosio}. Therefore, there exists a Borel family $\{\gamma_y^{1/k}\}$ in $\Gamma_T$ and a set $N_k$ of vanishing Lebesgue measure such that for every $y\in \T^d-N_k$, we have $\delta_{\gamma_{y}^{1/k}}=(B^{1/k})_\#\eta_{0,y}$. Define $$N:=\bigcup_{k\in\N} N_k,$$ which is of vanishing Lebesgue measure. By Lemma \ref{lem_weak_conv}, we have that $(B^{1/k})_\#\eta_{0,y}$ converges narrowly to $\eta_{0,y}$ as $k\to+\infty$ for every $y\in Y$. So there exists a Borel family $\{\gamma_y\}$ in $\Gamma_T$ such that for every $y\in \T^d-N$, the probability measure $\delta_{\gamma_{y}^{1/k}}$ converges narrowly to $\delta_{\gamma_y}$ as $k\to+\infty$ and $\eta_{0,y}=\delta_{\gamma_y}$. The thesis follows. 
\end{proof}

As a consequence, we have the following lemma. 
\begin{lemma}\label{prop_unique_incompressible}
	Consider a divergence-free vector field $b:[0,T]\times\T^d\to \R^d$ and $p,q\in (1,+\infty]$ satisfying \eqref{eqn_prodi_serrin}. Assume that $b\in L^1_{loc}((0,T];BV(\T^d;\R^d))\cap L^q_p(T)$. 
	Then there exists a unique $\eta$ in $\mathcal{P}(\Gamma_T)$ such that
	\begin{enumerate} 
		\item	$\eta$ is concentrated on integral curves of $b$, i.e. 
		\begin{equation}
		\int_{\Gamma_T}\Big|\gamma(t)-\gamma(0)-\int_0^tb(s,\gamma(s))ds\Big|\eta(d\gamma) \qquad\forall t\in[0,T];
		\end{equation}
		\item for every $t\in[0,T]$, we have $(e_t)_\#\eta=\Leb{d}$. 
	\end{enumerate} 
\end{lemma}
\begin{proof} 
\textbf{Step 1.} (Existence)
This follows directly from Proposition \ref{prop_vanishing_noise}.

\bigskip 
\textbf{Step 2.} (Uniqueness)
Let $\eta$ be a probability measure in $\mathcal{P}(\Gamma_T)$ satisfying $(i)$ and $(ii)$. Let $\t>0$ and recall definitions \eqref{eqn_def_S} and \eqref{eqn_def_b^t}. 
The probability measure $(S^\t)_\#\eta$ satisfies $(i)$ and $(ii)$ with $b$ replaced by $b^\tau$. Observe that $b^\t$ belongs to $L^1((0,T);BV(\T^d;\R^d))$ and is divergence-free, so in view of Theorem \ref{thm_ambrosio}, the probability measure $(S^\t)_\#\eta$ is uniquely determined. By \Cref{lem_weak_conv}, we have that $(S^\tau)_\#\eta$ converges narrowly to $\eta$ as $\t\downarrow 0$. Thus $\eta$ is uniquely determined. 
%In view of \Cref{thm_ambrosio}, there exists a unique $\eta^\tau$ in $\mathcal{P}(\Gamma_T)$ satisfying $(i)$ and $(ii)$ with $b$ replaced by $b^\tau$, and concentrated on Lipschitz curves with Lipschitz constant bounded by $\|b\|_{L^\infty_{t,x}}$. Therefore, the family $\{\eta^\tau\}_{\tau\in(0,T)}$ is tight and by Prokhorov's theorem, it admits a limit point $\eta$ in for narrow convergence. 
\end{proof} 
%\section{Vanishing noise limit} 
%\subsection{Proof of Theorem \ref{thm_main}}

We can now prove uniqueness of the zero-noise flow and $(i)$ of Theorem \ref{thm_main}. 

\begin{proof} [Proof of $(i)$ of Theorem \ref{thm_main}]
In view of  \Cref{prop_unique_incompressible}, there exists a unique probability measure $\eta$ satisfying $(i)$ concentration on integral curves of $b$ and $(ii)$ incompressibility. In view of \Cref{rmk_unique_vanishing_noise}, $\eta$ is the unique zero-noise flow of $b$. 

Let us now prove $(i)$. 
Let $a$ be a Borel vector field such that $a=b$ $\Leb{d+1}$-a.e.. Then there exists a zero-noise flow $ \bar{\eta}$ of $a$ by \Cref{prop_vanishing_noise}, which is concentrated on integral curves of $a$ and satisfies $(e_t)_\#\bar\eta=\Leb{d}$ for every $t\in[0,T]$. Note that for every $t\in[0,T]$, we have 
\begin{equation}
\int_{\Gamma_T}\Big| \gamma(t)-\gamma(0)-\int_0^t b(s,\gamma(s))ds\Big|\bar\eta(d\gamma)=\int_{\Gamma_T}\Big| \gamma(t)-\gamma(0)-\int_0^t a(s,\gamma(s))ds\Big|\bar\eta(d\gamma).
\end{equation}
Thus $\bar\eta$ is concentrated on integral curves of $b$ and by \Cref{prop_unique_incompressible}, this implies $\eta=\bar\eta$, and proves the thesis. 
\end{proof} 
\bigskip 

\section{Stochasticity of the zero-noise flow} \label{sec_stoch}
In this section we prove part $(ii)$ of Theorem \ref{thm_main}. Throughout this section, we shall take $\eta$ to be the unique zero-noise flow of $b$ satisfying the assumptions of \Cref{thm_main}. Let us gather the following remark, which insures that we can perform disintegrations on $\eta$ (see Section \ref{subsec_disintegration}). 
\begin{remark} \label{rmk_tightness}
	Consider a Borel vector field $b\in L^q_p(T)$ with $p$ and $q$ satisfying the Prodi-Serrin condition \eqref{eqn_prodi_serrin}, and a probability measure $\eta$ in $\mathcal{P}(\Gamma_T)$ concentrated on integral curves of $b$ such that $(e_s)_\#\eta=\Leb{d}$ for every $s\in[0,T]$. Then $\eta$ is tight. 
	Indeed, define the following subsets of $\Gamma_T$
	\begin{equation}
	K_N:=\Big\{ \gamma\in\Gamma_T : \int_0^T |\dot{\gamma}(s)|^2ds\leq N\Big\}. 
	\end{equation}
	By Ascoli's theorem, $K_N$ is compact in $\Gamma_T$. This then implies 
	\begin{equation}
	\int_{\Gamma_T}\int_0^T|\dot{\gamma}(s)|^2ds\eta(d\gamma)\leq \int_0^T\int_{\Gamma_T}|b(s,\gamma(s))|^2\eta(d\gamma)ds\leq \|b\|_{L^2((0,T)\times\T^d)}. 
	\end{equation}
	Therefore, by the Markov inequality, for every $N\in\N$, we have $$\eta(K_N^c)\leq \frac{|b\|_{L^2((0,T)\times\T^d)}}{N},$$ 
	and since $b$ is in $L^2((0,T)\times \T^d;\R^d)$, this implies tightness of $\eta$ in $\mathcal{P}(\Gamma_T)$. 
\end{remark} 
\subsection{Disintegration of $\eta$ at the final time}
 Let $\{\eta_{T,x}\}$ be a disintegration of $\eta$ with respect to $e_T$ and $\Leb{d}$. In this paragraph, we will give a representation for $\eta_{T,x}$ in terms on the backward integral curves of $b$. 
Define $\tilde b(t,x):=-b(T-t,x)$ and observe that this vector field lies in $L^1_{loc}([0,T);BV(\T^d;\R^d))$. 
By \Cref{thm_ambrosio}, there exists a unique $\tilde \eta$ in $\mathcal{P}(\Gamma_T)$ concentrated on integral curves of $\tilde b$ and such that $(e_t)_\#\tilde\eta=\Leb{d}$ for every $t\in[0,T]$. Furthermore, there exists a Borel family $\{\tilde\gamma_{y}\}$ in $\Gamma_T$ such that 
\begin{equation} \label{eqn_disint_tilde_eta} 
\tilde\eta=\int_{\T^d}\delta_{\tilde\gamma_{y}}dy.
\end{equation} 
Define the involution $$\beta:\Gamma_T\ni \gamma(\cdot)\longmapsto \gamma(T-\cdot)\in \Gamma_T.$$
%Notice also that for every $N\in\N$, the restriction $$\beta:K_N\ni \gamma(\cdot)\longmapsto \gamma(T-\cdot)\in  K_N$$ is also an involution. 
Define further the Borel family in $\Gamma_T$
\begin{equation} \label{eqn_def_gamma}
\gamma_y=\beta(\tilde\gamma_y) \qquad\forall y\in \T^d. 
\end{equation} 
We then have the following lemma. 
\begin{lemma}\label{lem_disint_final}
For $\Leb{d}$-a.e. $y\in\T^d$, we have $\eta_{T,y}=\delta_{\gamma_y}$.
\end{lemma}
\begin{proof}
	\textbf{Step 1.} Let us show that for every Borel set $A\subset \Gamma_T$, we have 
	\begin{equation}
	\eta(A)=\tilde\eta(\beta(A)). 
	\end{equation}
	Define the probability measure $\mu\in \mathcal{P}(\Gamma_T)$ by
		\begin{equation}\label{eqn_def_mu}
	\mu(A)=\tilde\eta(\beta(A)) 
	\end{equation}
	for every Borel set $A\subset \Gamma_T.$
	Clearly for every $t\in[0,T]$, we have $(e_t)_\#\mu=(e_{T-t})_\#\tilde\eta=\Leb{d}$. 
	\begin{comment}
\textbf{Step 1.} Let us show that for every $\Phi \in C_b(K_N)$, we have 
\begin{equation}\label{eqn_eta_tilde_eta}
\int_{\Gamma_T}\Phi(\gamma)\eta(d \gamma)=\int_{\Gamma_T}\Phi(\beta(\gamma))\tilde\eta(d\gamma). 
\end{equation}
By the Riesz representation theorem, there exists a finite positive measure $\mu_N$ in $\mathcal{P}(K_N)$ such that
\begin{equation}\label{eqn_mu} 
\int_{\Gamma_T}\Phi(\gamma)\mu_N(d\gamma)=\int_{\Gamma_T}\Phi(\beta(\gamma))\tilde\eta(d\gamma).
\end{equation}
For every $t\in [0,T]$ and every $\phi \in C(\T^d)$, we have $\phi \circ e_t\circ \beta=\phi \circ 
e_{T-t}.$ Since $(e_t)_\#\tilde \eta=\Leb{d}$, this implies that $(e_t)_\#\mu =\Leb{d}$.
\end{comment} 
Let us show that $\mu$ is concentrated on integral curves of $b$, i.e. 
\begin{equation}
\int_{\Gamma_T}\Big|\gamma(t)-\gamma(0)-\int_0^t b(s,\gamma(s))ds\Big|\mu(d\gamma)=0 \qquad\forall t\in [0,T]. 
\end{equation}
Let $\e>0$, $t\in[0,T]$ and $c:[0,T]\times\T^d\to \R^d$ be a continuous vector field such that $$\int_0^T\int_{\T^d}|c(s,x)-b(s,x)|dxds<\e.$$
Then, by continuity of the functional
$$\Gamma_T\ni \gamma\longmapsto \Big|\gamma(t)-\gamma(0)+\int_0^tc(T-s,\gamma(s))ds\Big|\in\R^+,$$
we have by \eqref{eqn_def_mu} that 
\begin{equation}
\int_{\Gamma_T}\Big| \gamma(t)-\gamma(0)+\int_0^t c(T-s,\gamma(s))ds\Big|\tilde\eta(d\gamma)=\int_{\Gamma_T}\Big| \gamma(t)-\gamma(0)-\int_0^t c(s,\gamma(s))ds\Big|\mu(d\gamma). 
\end{equation}
Since $\tilde \eta$ is concentrated on integral curves of $\tilde b$, we have 
\begin{equation}
\int_{\Gamma_T}\Big|\gamma(t)-\gamma(0)+\int_0^t b(T-s,\gamma(s))ds\Big|\tilde\eta(d\gamma)=0.
\end{equation}
Thus, we have
\begin{equation}
\begin{split}
&\int_{\Gamma_T}\Big| \gamma(t)-\gamma(0)-\int_0^t b(s,\gamma(s))ds\Big|\mu(d\gamma)\\
&\leq \int_{\Gamma_T}\Big| \gamma(t)-\gamma(0)-\int_0^t c(s,\gamma(s))ds\Big|\mu(d\gamma)+\int_0^T\int_{\T^d}|c(s,x)-b(s,x)|dxds\\
&=\int_{\Gamma_T}\Big| \gamma(t)-\gamma(0)+\int_0^t c(T-s,\gamma(s))ds\Big|\tilde\eta(d\gamma)+\int_0^T\int_{\T^d}|c(s,x)-b(s,x)|dxds\\
&\leq \int_{\Gamma_T}\Big|\gamma(t)-\gamma(0)+\int_0^t b(T-s,\gamma(s))ds\Big|\tilde\eta(d\gamma)+2\int_0^T\int_{\T^d}|c(s,x)-b(s,x)|dxds<2\e.
\end{split}
\end{equation}
As $\e$ and $t$ were arbitrary, this shows that $\mu$ is concentrated on integral curves of $b$. 
In view of \Cref{thm_ambrosio}, this shows that $\mu=\eta$.% and by \eqref{eqn_mu}, it thus holds \eqref{eqn_eta_tilde_eta}. 

\bigskip 

\textbf{Step 2.} Let $\{\eta_{T,y}\}$ be a disintegration of $\eta$ with respect to $e_T$ and $\Leb{d}$.  In view of \Cref{lem_disintegration_push_forward} with $f=e_T$, $g=\beta$ and $h=e_0$, it holds that
$\{\beta_\#\eta_{T,y}\}$ is a disintegration of $\beta_\#\eta$ with respect to $e_0$ and $\Leb{d}$, and since $\beta$ is an involution we have $\beta_\#\eta=\tilde\eta$ by \textbf{Step 1}. Therefore by \eqref{eqn_disint_tilde_eta} and essential uniqueness of the disintegration, we have $\beta_\#\eta_{T,y}=\tilde\eta_{0,x}=\delta_{\tilde\gamma_y}$ for $\Leb{d}$-a.e. $y\in \T^d$, whence using again that $\beta$ is an involution, we have $\eta_{T,y}=\delta_{\gamma_y}$ for $\Leb{d}$-a.e. $y\in\T^d$. This proves the thesis. 
\begin{comment} 
Then using \eqref{eqn_int_disintegration} we have 
\begin{equation}
\int_{\Gamma_T}\Phi(\gamma)\phi(\gamma(T))\eta(d\gamma)=\int_{\Gamma_T}\Phi(\gamma)\phi(\gamma(T))\eta_{T,y}(d\gamma)dy=\int_{\Gamma_T}\Phi(\gamma)\eta_{T,y}(d\gamma)\phi(y)dy.
\end{equation}
In view of \eqref{eqn_eta_tilde_eta} and of \eqref{eqn_disint_tilde_eta}, using \eqref{eqn_int_disintegration} we also have 
\begin{equation}
\int_{\Gamma_T}\Phi(\gamma)\phi(\gamma(T))\eta(d\gamma)=\int_{\Gamma_T}\Phi(\beta(\gamma))\phi(\gamma(0))\tilde\eta(d\gamma)=\int_{\Gamma_T} \Phi(\beta(\gamma))\delta_{\tilde\gamma_y}(d\gamma)\phi(y)dy. 
\end{equation}
Thus, we have
\begin{equation}
\int_{\Gamma_T} \Phi(\beta(\gamma))\delta_{\tilde\gamma_y}(d\gamma)\phi(y)dy=\int_{\Gamma_T}\Phi(\gamma)\eta_{T,y}(d\gamma)\phi(y)dy.
\end{equation}
As $\phi$ was arbitrary, this implies that there exists a set of vanishing Lebesgue measure $N_\Phi$ such that for every $x\in \T^d-N_\Phi$, we have 
\begin{equation}
\int_{\Gamma_T} \Phi(\beta(\gamma))\delta_{\tilde\gamma_y}(d\gamma)=\int_{\Gamma_T}\Phi(\gamma)\eta_{T,y}(d\gamma).
\end{equation}
Now set 
$$N:=\bigcup _{\Phi\in \mathcal{D}}N_\Phi,$$ which also has vanishing Lebesgue measure and recall the definition of $\gamma_y$ in \eqref{eqn_def_gamma}. 
Then by density of $\mathcal{D}$, for every $\Phi \in C_b(\Gamma)$ and every $x\in \T^d-N$, we have \begin{equation}
\int_{\Gamma_T}\Phi(\gamma)\delta_{\gamma_y}(d\gamma)=\int_{\Gamma_T}\Phi(\gamma)\eta_{T,y}(d\gamma),
\end{equation} 
which by the Riesz representation theorem proves the thesis.
\end{comment} 
\end{proof} 

\subsection{Disintegration of $\eta$ at the initial time}\label{sec_disint}
%we define the map $$e_{u,v}:\Gamma \ni \gamma \longmapsto (\gamma(u),\gamma(v))\in \T^d\times\T^d.$$ 
We define the measure $\nu =(e_0,e_T)_\#\eta$ on $\T^d\times\T^d$.
For every $x\in \T^d,$ we define the family of measures on $\Gamma_T$
\begin{equation}\label{eqn_def_ddelta_x_gamma}
\delta_{x,\gamma_{y}}:=\left\{ 
\begin{split} 
\delta_{\gamma_{y}}\qquad &\text{if} \quad \gamma_{y}(0)=x,\\
0\qquad &\text{if}\quad \gamma_{y}(0)\neq x.
\end{split}
\right. 
\end{equation}
Define the projection maps $$\pi_0 :\T^d\times\T^d\ni (x,y)\longmapsto x\in \T^d,$$ and $$\pi_1 :\T^d\times\T^d\ni (x,y)\longmapsto y\in \T^d.$$ Let $\{\nu_x\}$ be a disintegration of $\nu$ with respect to $\pi_0$ and $\Leb{d}$.

We will now give an expression for disintegrations of $\eta$ with respect to $e_0$ and $\Leb{d}$ in terms of $\{\nu_x\}$. Throughout this section $\{\eta_{0,x}\}$ is a disintegration of $\eta$ with respect to $e_0$ and $\Leb{d}$. We then define $\tilde\nu_x :=(\pi_1)_\#\nu_x$ for every $x\in \T^d$. We denote by $\{\gamma_y\}$ the Borel family in $\Gamma_T$ given by \Cref{lem_disint_final}. We also define the probability measure 
\begin{equation} \label{eqn_def_nu_y}
\nu_y:=\delta_{(\gamma_{y}(0),y)},
\end{equation}
on $\T^d\times \T^d$ for every $y\in \T^d$. 
\begin{lemma} \label{lem_disint_nu}
	The family	$\{\nu_y\}$ is a disintegration of $\nu$ with respect to $\pi_1$ and $\Leb{d}.$ 
\end{lemma}	
\begin{proof}
	It is clear that $\nu_y$ is supported on $\pi_1^{-1}(y)$. 
	By part $(i)$ of Theorem \ref{thm_main}, we know also that $\{ \delta_{\gamma_{y}} \}$ is a disintegration of $\eta$ with respect to $e_T$ and $\Leb{d}$. Therefore, for every Borel set $A$ in $\T^d\times \T^d$, we have $$\nu(A)=(e_0,e_T)_\#\eta(A)= \int_{\T^d} (e_0,e_T)_\#\delta_{\gamma_{y}} (A)dy=\int_{\T^d} \delta_{(\gamma_{y}(0),y)}(A)dy=\int_{\T^d}\nu_y(A)dy,$$
	which proves the thesis.
\end{proof}

\begin{lemma}\label{lem_disint_eeta_nu}
	The family $\{\delta_{x,\gamma_{y}}: x,y\in \T^d\}$ is a disintegration of $\eta$ with respect to $(e_{0},e_T)$ and $\nu$. 
\end{lemma} 

\begin{proof}
	It is clear that $\delta_{x,\gamma_{y}}$ is supported on $(e_0,e_T)^{-1} (x,y)$. 
	%Consider the disintegration $\{\ddelta_{\gamma_{1,y}(0)\}$ of $\nu$. 
	For every Borel set $A$ contained in $\Gamma_T$, we have
	\begin{equation}
	\begin{split} 
	\int_{\T^d\times\T^d}\delta_{x,\gamma_{y}}(A) d\nu(x,y)&=\int_{\T^d}\int_{\T^d\times \{y\}} \delta_{x,\gamma_{y}}(A) d\nu_ydy\\
	&=\int_{\T^d} \delta_{\gamma_{y}}(A)dy\\
	&=\eta(A),
	\end{split}
	\end{equation}
	where in the first equality we have used Lemma \ref{lem_disint_nu}, as well as \eqref{eqn_int_disintegration}. In the second equality we have used the definition \eqref{eqn_def_nu_y} of $\nu_y$ and the definition \eqref{eqn_def_ddelta_x_gamma} of $\delta_{x,\gamma_{y}}$, and in the last equality we have \Cref{lem_disint_final}. 
	This proves the claim.
\end{proof}
\begin{lemma}\label{lem_disint_nu_x}
	For $\Leb{d}$-a.e. $x\in \T^d$, we have
	\begin{equation}
	\eta_{0,x}=\int_{\T^d}\delta_{\gamma_{y}}d\tilde\nu_x(y).
	\end{equation}
\end{lemma}

\begin{proof}
	Let $B$ be a Borel set contained in $\T^d$. As $\Gamma_T$ is separable, its Borel $\s$-algebra is generated by a countable family $\mathscr{G}$. Let $A\in \mathscr{G}$. We then have
	\begin{equation}
	\begin{split}
	\int_B \eta_{0,x}(A)dx
	&\overset{1}{=}\int_{\T^d} \eta_{0,x}(A\cap\{\gamma(0)\in B\})dx\\
	&\overset{2}{=}\eta(A\cap\{\gamma(0)\in B\})\\
	&\overset{3}{=}\int_{\T^d\times\T^d} \delta_{x,\gamma_{y}}(A\cap\{\gamma(0)\in B\})d\nu(x,y)\\
	&\overset{4}{=}\int_B\Big[ \int_{\{x\}\times\T^d} \delta_{x,\gamma_{y}}(A) d\nu_x\Big]dx \\
	&\overset{5}{=}\int_B\Big[\int_{\T^d}\delta_{\gamma_{y}}(A) d\tilde\nu_x(y)\Big]dx.
	\end{split} 
	\end{equation}
	In equality $1$, we have used that $\eta_{0,x}$ is supported on $\{\gamma(0)=x\}$ for every $x\in\T^d$. In equality $2$, we have used that $\{\eta_{0,x}\}$ is a disintegration of of $\eta$ with respect to $e_0$ and $\Leb{d}$. In equality $3$, we have used Lemma \ref{lem_disint_eeta_nu}. In equality $4$, we have used that $\{\nu_x\}$ is a disintegration of $\nu$ with respect to $\pi_0$ and $\Leb{d}$, equation \eqref{eqn_int_disintegration}, as well as the fact that $\delta_{x,\gamma_{y}}(A\cap \{\gamma(0)\in B\})=0$ if $x\notin B$ by definition. 
	In equality $5$, we have used the definition of $\delta_{x,\gamma_{y}}$ as well as the definition of $\{\tilde\nu_x\}$. 
	As $B$ was an arbitrary Borel set in $\T^d$, there exists a set $N_A$ of vanishing Lebesgue measure such that for every $x\in \T^d-N_A$, we have
	%and every Borel set $A$ contained in $\Gamma$ such that $A\subset \{\gamma(0)=x\}$  
	$$\eta_{0,x}(A)=\int_{\T^d}\delta_{\gamma_{y}}(A) d\tilde\nu_x(y).$$
	Now define $$N:=\bigcup_{A\in \mathscr{G}}N_A,$$ 
	which is a set of vanishing Lebesgue measure. 
	Then, for every $A\in \mathscr{G}$ and every $x\in \T^d- N$, we have 
	\begin{equation}
	\eta_{0,x}(A)=\int_{\T^d}\delta_{\gamma_{y}}(A) d\tilde\nu_x(y). 
	\end{equation}
	%As $\eeta_{0,x}$ is supported in $\{\gamma(0)=x\}$, it follows that for every Borel set $A$ contained in $\Gamma$, we have 
	%\begin{equation}
	%\eeta_{0,x}(A)=\int_{\T^d}\ddelta_{\gamma_{1,y}}(A)\nu_x(dy),
	%\end{equation}
	As $\mathscr{G}$ generates the Borel $\s$-algebra of $\Gamma_T$, the thesis is proved. 
\end{proof}
We have almost finished the proof of $(ii)$ of \Cref{thm_main}. We conclude by analysing the case $b=b_{DP}$. 
\subsection{Disintegration of $\eta$ for the Depauw vector field} 

Let us now show that for the vector field $b_{DP}:[0,T]\times \T^2\to \R^2$ constructed by Depauw \cite{Depauw}, the family probability measures $\tilde\nu_x$ are not Dirac masses for $\Leb{2}$-a.e. $x\in \T^2$. This will conclude the proof of Theorem \ref{thm_main}. This paragraph draws from \cite{Pitcho24Int}. 

\bigskip 
Let $\eta$ be the unique zero-noise flow of $b_{DP}$. 
From \Cref{sec_const_depauw}, we have two bounded densities $\rho^B$ and $\rho^W$ in $C([0,T];w^*-L^\infty(\T^2))$ solving 
\begin{equation}\label{eqn_pde}\tag{PDE}
\div_{t,x}\rho(1,b_{DP})=0\quad\text{in the sense of distributions on  $[0,T]\times\T^d$,}
\end{equation}
which further satisfy
\begin{enumerate} 
	%\item  $\rho^1(1,\bb_{DP})$ and $\rho^2(1,\bb_{DP})$ solve \eqref{eqn_pde};
	\item $\rho^B(0,\cdot)=1/2=\rho^W(0,\cdot)$;
	\item $\rho^B(t,\cdot)+\rho^W(t,\cdot)=1$ for every $t\in[0,T]$;
	\item $\supp\rho^B(T,\cdot)\cup \supp\rho^W(T,\cdot)=\T^2$;
	\item $\supp\rho^B(T,\cdot)\cap \supp\rho^W(T,\cdot)$ is of vanishing Lebesgue measure. 
\end{enumerate} 

By Smirnov's representation for normal 1-currents \cite{Smirnov94} or Ambrosio's superposition principle \cite{ambrosiocrippaedi}, there exist probability measures $\eta^B$ and $\eta^W$ in $\mathcal{P}(\Gamma_T)$ concentrated on integral curves of $b_{DP}$ such that for every $t\in[0,T]$, we have $$(e_t)_\#\eta^B=\rho^B(t,\cdot)\Leb{d}\quad\text{and}\quad (e_t)_\#\eta^W=\rho^W(t,\cdot)\Leb{d}.$$
Let $\nu^B$ and $\nu^W$ be two probability measures given by $\nu^B=(e_0,e_T)_\#\eta^B$ and $\nu^W=(e_0,e_T)_\#\eta^W$. Let $\{\eta_{0,x}^B\}$ be a disintegration of $\eta^B$ with respect to $e_0$ and $\Leb{2}$, and let $\{\nu_x^B\}$ be a disintegration of $\nu^B$ with respect to $\pi_0$ and $\Leb{2}$. Similarly, let $\{\eta_{0,x}^W\}$ be a disintegration of $\eta^W$ with respect to $e_0$ and $\Leb{2}$, and let $\{\nu_x^W\}$ be a disintegration of $\nu^W$ with respect to $\pi_0$ and $\Leb{2}$.
Note that by definition of $\nu^B$ and $\nu^W$, we have
$$(\pi_1)_\#\nu^B=(e_T)_\#\eta^B\qquad\text{and}\qquad (\pi_1)_\#\nu^W=(e_T)_\#\eta^W.$$ 
This clearly implies that for $\Leb{2}$-a.e. $x\in\T^2$, we have
\begin{equation}\label{eqn_disint_nu}
(\pi_1)_\#\nu^B_x=(e_T)_\#\eta_{0,x}^B \qquad\text{and} \qquad (\pi_1)_\#\nu^W_x=(e_T)_\#\eta_{0,x}^W.
\end{equation}
Therefore, we have 
\begin{equation*}
\begin{split} 
\int_{\T^2}(\pi_1)_\#\nu^B_x(\supp \rho^W(T,\cdot))dx&=\int_{\T^2}(e_T)_\#\eta^B_{0,x}(\supp\rho^W(T,\cdot))dx\\
&=(e_T)_\#\eta^B(\supp \rho^W(T,\cdot))\\
&=\int_{\supp\rho^W(T,\cdot)}\rho^B(T,x)dx\\
&=0.
\end{split} 
\end{equation*}
Similarly, we have 
\begin{equation*}
\int_{\T^2}(\pi_1)_\#\nu^W_x(\supp \rho^B(T,\cdot))dx=0.
\end{equation*}
Therefore, for $\Leb{2}$-a.e. $x\in\T^2$, in view of property $(iv)$ above, the probability measures $(\pi_1)_\#\nu^W_x$ and $(\pi_1)_\#\nu^B_x$ are mutually singular.

Also, by property $(ii)$ above, we have $$(e_t)_\#\frac{1}{2}(\eta^B+\eta^W)=\Leb{d}$$ for every $t\in[0,T],$ whence
\begin{equation}
\eta=\frac{1}{2}(\eta^B+\eta^W).
\end{equation}
By essential uniqueness of the disintegration, we therefore have for $\Leb{2}$-a.e. $x\in \T^2$  
\begin{equation*}
\nu_x=\frac{1}{2}(\nu^W_x+\nu^B_x).
\end{equation*}
Therefore, for $\Leb{2}$-a.e. $x\in\T^2$, we have
\begin{equation*}
\tilde\nu_x=(\pi_1)_\#\nu_x=\frac{1}{2}(((\pi_1)_\#\nu^W_x+(\pi_1)_\#\nu^B_x).
\end{equation*}
whereby for $\Leb{2}$-a.e. $x\in \T^2$ the probability measure $\tilde\nu_x$ is not a Dirac mass. 
This concludes the proof of $(ii)$ of Theorem \ref{thm_main}. 
\bigskip 

\subsection{Construction of the Depauw vector field}\label{sec_const_depauw}
We construct the bounded, divergence-free vector field $b_{DP}:[0,T]\times\T^2\to \R^2$ of Depauw from \cite{Depauw}, as well as two densities $\rho^W,\rho^B:[0,T]\times \T^2\to \R^+$ such that the vector fields $\rho^W(1,b_{DP})$ and $\rho^B(1,b_{DP}) $ solve \eqref{eqn_pde}, and have the following properties:
\begin{enumerate} 
	\item $\rho^B(0,\cdot)=1/2=\rho^W(0,\cdot)$;
	\item $\rho^B(t,\cdot)+\rho^W(t,\cdot)=1$ for every $t\in[0,T]$;
	\item $\supp\rho^B(T,\cdot)\cup \supp\rho^W(T,\cdot)=\T^2$;
	\item $\supp\rho^B(T,\cdot)\cap \supp\rho^W(T,\cdot)$ is of vanishing Lebesgue measure.
\end{enumerate} 
We follow closely the construction of a similar vector field given in \cite{DeLellis_Giri22}. 
\bigskip 

Introduce the following two lattices on $\mathbb R^2$, namely $\mathcal{L}^1 := \mathbb Z^2\subset \mathbb R^2$ and $\mathcal{L}^2:=\mathbb Z^2 + (\frac{1}{2}, \frac{1}{2})\subset \mathbb R^2$. To each lattice, associate a subdivision of the plane into squares, which have vertices lying in the corresponding lattices, which we denote by $\mathcal{S}^1$ and $\mathcal{S}^2$. Then consider the rescaled lattices $\mathcal{L}^1_k:= 2^{-k} \mathbb{Z}^2$ and $\mathcal{L}^2_k := (2^{-k-1},2^{-k-1})+2^{-k} \mathbb Z^2$ and the corresponding square subdivision of $\mathbb Z^2$, respectively $\mathcal{S}^1_k$ and $\mathcal{S}^2_k$. Observe that the centres of the squares $\mathcal{S}^1_k$ are elements of $\mathcal{L}^2_k$ and viceversa.

%if we index the squares of $\mathcal{S}^1$ with $(k,j)$, where $(k+\frac{1}{2}, j+\frac{1}{2})\in \mathcal{L}^2$ is the centre of the corresponding square, then $\rho^B_{final}$ vanishes on the squares for which $k+j$ is even, while it is identically equal to $1$ on squares for which $k+j$ is odd.

Next, define the following $2$-dimensional autonomous vector field:
\[
w(x) =
\begin{cases}
(0, 4x_1)^t\text{ , if }1/2 > |x_1| > |x_2| \\
(-4x_2, 0)^t\text{ , if }1/2 > |x_2| > |x_1| \\
(0, 0)^t\text{ , otherwise.} \\
\end{cases}
\]
$w$ is a bounded, divergence-free vector field, whose derivative is a finite matrix-valued Radon measure given by
\begin{equation*}
\begin{split} 
Dw(x_1,x_2) = 
&\begin{pmatrix}
0 & 0 \\
4{\rm sgn}(x_1) & 0 
\end{pmatrix} 
\Leb{2}\lfloor_{\{|x_2|<|x_1|<1/2\}} +
\begin{pmatrix}
0 & -4{\rm sgn}(x_2) \\
0 & 0 
\end{pmatrix} 
\Leb{2}\lfloor_{\{|x_1|<|x_2|<1/2\}}\\
%\begin{}
%&+
%\begin{pmatrix}
%4x_2 & -4x_2 \\
%4x_1 & -4x_1
%\end{pmatrix} 
%\mathcal{H}^{d-1}\lfloor_{\{x_1=x_2,0<x_1,x_2\leq 1/2\}}\\
%&+
%\begin{pmatrix}
%-4x_2 & -4x_2 \\
%-4x_1 & -4x_1
%\end{pmatrix} 
%\mathcal{H}^{d-1}\lfloor_{\{x_1=x_2,0<-x_1,x_2\leq 1/2\}}\\
%&+
%\begin{pmatrix}
%4x_2 & 4x_2 \\
%4x_1 & 4x_1
%\end{pmatrix} 
%\mathcal{H}^{d-1}\lfloor_{\{x_1=x_2,0<x_1,-x_2\leq 1/2\}}\\
%&+
%\begin{pmatrix}
%-4x_2 & 4x_2 \\
%-4x_1 & 4x_1
%\end{pmatrix} 
%\mathcal{H}^{d-1}\lfloor_{\{x_1=x_2,0<-x_1,-x_2\leq 1/2\}}\\
&+\begin{pmatrix}
4x_2{\rm sgn}(x_1) & -4x_2{\rm sgn}(x_2) \\
4x_1{\rm sgn}(x_1) & -4x_1{\rm sgn}(x_2)
\end{pmatrix} 
\mathscr{H}^{1}\lfloor_{\{x_1=x_2,0<|x_1|,|x_2|\leq 1/2\}}\\
\end{split} 
\end{equation*}

% (c.f. Section 7 of \cite{HSSS}). 
Periodise $w$ by defining $\Lambda = \{(y_1, y_2) \in \mathbb{Z}^2 : y_1 + y_2 \text{ is even}\}$ and setting 
\[
u(x) =T \sum_{y \in \Lambda} w(x-y)\, .
\]
%Note that that $\ww$ is supported in one square of $\mathcal{S}^2$ and thus the periodization consists of filling half the squares of $\mathcal{S}^2$ with copies of $w$, while leaving the field identically equal to $0$ in the remaining squares. The ``filled'' and ``empty'' squares form likewise a chessboard pattern. 

Even though $u$ is non-smooth, it is in $BV_{loc}(\R^2;\R^2)$. 
\begin{comment} 
Choose a standard mollifier $\zeta\in C^\infty_c(\R^d)$, and consider the regularisation $(u^k)_{k\in\N}$ of $\uu$ given by $\uu^k=\uu\star \zeta^k$, and which satisfies the hypothesis of Proposition \ref{prop_well_posed}.  As $\uu$ satisfies the hypothesis of Theorem \ref{thm_ambrosio}, the Cauchy problem \eqref{eqPDE} is well-posed along the regularisation $(\uu^k)_{k\in\N}$. Therefore, to determine the unique bounded weak solution of \eqref{eqPDE} along $\uu$ with initial datum $\bar\rho$, it suffices to determine the unique bounded weak solution of \eqref{eqPDE} along $\uu^k$ with initial datum $\bar\rho$ for $k$ arbitrarly large.
\end{comment} 
By the theory of regular Lagrangian flows (see for instance \cite{ambrosiocrippaedi}), there exists a unique incompressible almost everywhere defined flow $X$ along $u$ can be described explicitely. %Except on the closed union of lines $$\mathcal{L}:=\bigcup_{y\in\Lambda}\{x_2=\pm x_1 +y\},$$
%$\uu$ is smooth. 
%By the classical Cauchy-Lipschitz theory, for $x\in \R^2/\mathcal{L}$, there exists times $T^-<0<T^+$ such that $\partial_t \XX(t,0,x)=\uu(\XX(t,0,x))$ admits a unique solution of $(T^-,T^+)$ such that $\XX(0,0,x)=x$, and $T^-$ and $T^+$ are respectively the first negative and positive times where $\XX(t,x)$ hits $\mathcal{L}$.
% By extending continuously, this flow can be made global.
\begin{itemize}
	\item[(R)] The map $X (t,\cdot)$ is Lipschitz on each square $S$ of $\mathcal{S}^2$ and $X (T/2, \cdot)$ is a clockwise rotation of $\pi/2$ radians of the ``filled'' $S$, while it is the identity on the ``empty ones''. In particular for every $j\geq 1,$ $X (T/2,\cdot)$ maps an element of $\mathcal{S}^1_j$ rigidly onto another element of $\mathcal{S}^1_j$. For $j=1$ we can be more specific. Each $S\in \mathcal{S}^2$ is formed precisely by $4$ squares of $\mathcal{S}^1_1$: in the case of ``filled'' $S$ the $4$ squares are permuted in a $4$-cycle clockwise, while in the case of ``empty'' $S$ the $4$ squares are kept fixed.  
\end{itemize}

\begin{figure}[h]
	\centering
	\subfloat{
		\begin{tikzpicture}
		\clip(-1.5,-1.5) rectangle (1.5,1.5);
		\foreach \x in {-1,...,2} \foreach \y in {-1,...,2}
		{
			\pgfmathparse{mod(\x+\y,2) ? "black" : "white"}
			\edef\nu{\pgfmathresult}
			\path[fill=\nu, opacity=0.5] (\x-1,\y-1) rectangle ++ (1,1);
		}
		\draw (-1.5,0.5) -- (1.5,0.5);
		\draw (-1.5,-0.5) -- (1.5,-0.5);
		\draw (0.5, -1.5) -- (0.5, 1.5);
		\draw (-0.5, -1.5) -- (-0.5, 1.5);
		\foreach \a in {1,...,2}
		{
			\draw[ultra thick, ->] (\a/5,-\a/5) -- (\a/5,\a/5); 
			\draw[ultra thick, ->] (\a/5,\a/5) -- (-\a/5,\a/5);
			\draw[ultra thick, ->] (-\a/5,\a/5) -- (-\a/5,-\a/5);
			\draw[ultra thick, ->] (-\a/5,-\a/5) -- (\a/5,-\a/5);
		}
		\foreach \a in {1,...,2}
		{
			\draw[ultra thick, ->] (\a/5+1,-\a/5+1) -- (\a/5+1,\a/5+1); 
			\draw[ultra thick, ->] (\a/5+1,\a/5+1) -- (-\a/5+1,\a/5+1);
			\draw[ultra thick, ->] (-\a/5+1,\a/5+1) -- (-\a/5+1,-\a/5+1);
			\draw[ultra thick, ->] (-\a/5+1,-\a/5+1) -- (\a/5+1,-\a/5+1);
		}
		\foreach \a in {1,...,2}
		{
			\draw[ultra thick, ->] (\a/5+1,-\a/5-1) -- (\a/5+1,\a/5-1); 
			\draw[ultra thick, ->] (\a/5+1,\a/5-1) -- (-\a/5+1,\a/5-1);
			\draw[ultra thick, ->] (-\a/5+1,\a/5-1) -- (-\a/5+1,-\a/5-1);
			\draw[ultra thick, ->] (-\a/5+1,-\a/5-1) -- (\a/5+1,-\a/5-1);
		}
		\foreach \a in {1,...,2}
		{
			\draw[ultra thick, ->] (\a/5-1,-\a/5-1) -- (\a/5-1,\a/5-1); 
			\draw[ultra thick, ->] (\a/5-1,\a/5-1) -- (-\a/5-1,\a/5-1);
			\draw[ultra thick, ->] (-\a/5-1,\a/5-1) -- (-\a/5-1,-\a/5-1);
			\draw[ultra thick, ->] (-\a/5-1,-\a/5-1) -- (\a/5-1,-\a/5-1);
		}
		\foreach \a in {1,...,2}
		{
			\draw[ultra thick, ->] (\a/5-1,-\a/5+1) -- (\a/5-1,\a/5+1); 
			\draw[ultra thick, ->] (\a/5-1,\a/5+1) -- (-\a/5-1,\a/5+1);
			\draw[ultra thick, ->] (-\a/5-1,\a/5+1) -- (-\a/5-1,-\a/5+1);
			\draw[ultra thick, ->] (-\a/5-1,-\a/5+1) -- (\a/5-1,-\a/5+1);
		}
		\draw (-2,-2) -- (2,2);
		\draw (-2,2) -- (2,-2);
		\draw (-1.5,0.5) -- (-0.5, 1.5);
		\draw (1.5,0.5) -- (0.5,1.5);
		\draw (1.5,-0.5) -- (0.5, -1.5);
		\draw (-1.5,-0.5) -- (-0.5,-1.5);
		\end{tikzpicture}%
	}
	\qquad
	\qquad
	\subfloat
	{
		\begin{tikzpicture}[scale=0.5]
		\foreach \x in {-2,...,3} \foreach \y in {-2,...,3}
		{
			\pgfmathparse{mod(\x+\y,2) ? "white" : "black"}
			\edef\nu{\pgfmathresult}
			\path[fill=\nu, opacity=0.5] (\x-1,\y-1) rectangle ++ (1,1);
		}
		\end{tikzpicture}%
	}
	\caption{Action of the flow of $u$ from $t=0$ to $t=T/2$. The shaded region denotes the set $\{\rho^B=1\}$. The figure is from \cite{DeLellis_Giri22}.}
\end{figure}
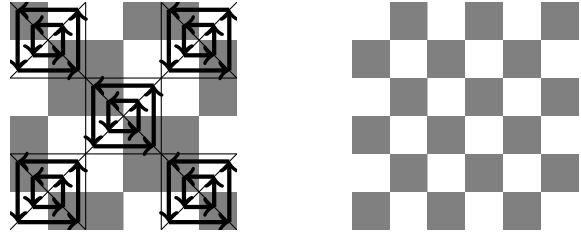

Let $\rho^B:[T/2,T]\times \R^2\to \R^+$ be the unique  density such that $\rho^B(1,u)$ solves \eqref{eqn_pde} and $\rho^B(T,\cdot)=\lfloor{x_1}\rfloor /2+ \lfloor{x_2}\rfloor/2 \ mod \ 2=:\bar\rho^B$. Then we have the following formula  $X(t,\cdot)_\#\bar\rho^B\Leb{d}=\rho^B(t,x)\Leb{d}$. 
Using property (R), we have 
\begin{equation}\label{e:refining}
\rho^B(T/2, x) = 1 - \bar\rho^B(2x) .
\end{equation}

We define $b_{DP}:[0,T]\times\R^2\to\R^2$ as follows. Set $b_{DP}(t, x) = u(x)$ for $T/2<t\leq T$ and $b_{DP}(t, x) = u(2^k x)$ for $T/2^{k+1}<t\leq T/2^{k}$. Let $\rho^B:[0,T]\times \R^2\to \R^+$ be the unique  density such that $\rho^B(1,u)$ solves \eqref{eqn_pde} with $\rho^B(T,\cdot)=\bar\rho^B(\cdot)$
% For $1<t<2$, we let $b(t,x) = -b(2-t,x)$. 
Moreover, using recursively the appropriately scaled version of \eqref{e:refining}, we can check that 
\begin{equation*} 
\rho^B(T/2^{k}, x) = \bar\rho^B (2^{k} x) \quad\text{for $k$ even,}\qquad \rho^B (T/2^{k}, x) =1- \bar\rho^B (2^{k} x)\quad\text{ for $k$ odd. }
\end{equation*} 
% and $\zeta (1-1/2^{2k+1},x) = 1 -\zeta_{in} (2^{2k+1} x)$. 

Define the density $\rho^W(t,x):=1-\rho^B(t,x)$. Then $\rho^W(1,b_{DP})$ also solves \eqref{eqn_pde}, by linearity. As the construction we have performed is $\Z^2$-periodic, we may consider $b_{DP}$, $\rho^W,$ and $\rho^B$ to be defined on $[0,T]\times \T^2$. 
Properties $(i)$-$(iv)$ follow from the construction. 

\bibliographystyle{alpha}
\bibliography{bibliografia}

\begin{bibdiv}
\begin{biblist}

\bib{ABC14}{article}{
      author={Alberti, Giovanni},
      author={Bianchini, Stefano},
      author={Crippa, Gianluca},
       title={A uniqueness result for the continuity equation in two
  dimensions},
        date={2014},
     journal={Journal of the European Mathematical Society},
      volume={16},
       pages={201\ndash 234},
}

\bib{AmbBV}{article}{
      author={Ambrosio, Luigi},
       title={Transport equation and {C}auchy problem for {B}{V} vector
  fields},
        date={2004},
     journal={Inventiones mathematicae},
      volume={158},
       pages={227\ndash 260},
}

\bib{ambrosiocrippaedi}{article}{
      author={Ambrosio, Luigi},
      author={Crippa, Gianluca},
       title={Continuity equations and {ODE} flows with non-smooth velocity},
        date={2014},
     journal={Proc. Roy. Soc. Edinburgh Sect. A},
      volume={144},
       pages={1191\ndash 1244},
}

\bib{baficobaldi82}{article}{
      author={Bafico, R.},
      author={Baldi, P.},
       title={Small random perturbations of {P}eano phenomena},
        date={1982},
     journal={Stochastics},
      volume={6},
       pages={279\ndash 292},
}

\bib{BianchiniBonicatto20}{article}{
      author={Bianchini, S.},
      author={Bonicatto, P.},
       title={A uniqueness result for the decomposition of vector fields in
  $\mathbb{R}^d$},
        date={2020},
     journal={Inventiones Mathematicae},
      volume={220},
       pages={255 \ndash  393},
}

\bib{BressanMazzolaNguyen-Markovian23}{article}{
      author={Bressan, Alberto},
      author={Mazzola, Marco},
      author={Nguyen, Khai~T.},
       title={Markovian solutions to discontinuous {ODE}s},
        date={2023},
     journal={Journal of Dynamics and Differential Equations},
      volume={35},
       pages={135\ndash 162},
}

\bib{BressanMazzolaNguyen-Diffusion23}{article}{
      author={Bressan, Alberto},
      author={Mazzola, Marco},
      author={Nguyen, Khai~T.},
       title={Diffusion approximations of markovian solutions to discontinuous
  {ODE}s},
        date={2024},
     journal={Journal of Dynamics and Differential Equations},
      volume={36},
       pages={1367\ndash 1404},
}

\bib{CiampaCrippaSpirito19}{article}{
      author={Ciampa, Gennaro},
      author={Crippa, Gianluca},
      author={Spirito, Stefano},
       title={Smooth approximation is not a selection principle for the
  transport equation with rough vector field},
        date={2019},
     journal={Calculus of Variations and Partial Differential Equations},
      volume={59},
}

\bib{colombo2022anomalous}{article}{
      author={Colombo, Maria},
      author={Crippa, Gianluca},
      author={Sorella, Massimo},
       title={Anomalous {D}issipation and {L}ack of {S}election in the
  {O}bukhov-{C}orrsin {T}heory of {S}calar {T}urbulence},
        date={2023},
     journal={Annals of PDE},
      volume={9},
}

\bib{DeLellis_Giri22}{article}{
      author={De~Lellis, Camillo},
      author={Giri, Vikram},
       title={Smoothing does not give a selection principle for transport
  equations with bounded autonomous fields},
        date={2022},
     journal={Ann. Math. Qu\'{e}.},
      volume={46},
       pages={27\ndash 39},
}

\bib{DelarueFlandoli14}{article}{
      author={Delarue, F.},
      author={Flandoli, F},
       title={The transition point in the zero noise limit for a 1{D} {P}eano
  example},
        date={2014},
     journal={Discrete and continuous dynamical systems},
      volume={34},
       pages={4071\ndash 4083},
}

\bib{Depauw}{article}{
      author={Depauw, Nicolas},
       title={Non unicit\'{e} des solutions born\'{e}es pour un champ de
  vecteurs {BV} en dehors d'un hyperplan},
        date={2003},
     journal={C. R. Math. Acad. Sci. Paris},
      volume={337},
       pages={249\ndash 252},
}

\bib{DPL89}{article}{
      author={DiPerna, R.~J.},
      author={Lions, P.-L.},
       title={Ordinary differential equations, transport theory and {S}obolev
  spaces},
        date={1989},
     journal={Inventiones Mathematicae},
      volume={98},
       pages={511\ndash 547},
}

\bib{FlandoliFedrizzi11}{article}{
      author={Fedrizzi, E.},
      author={Flandoli, F.},
       title={Pathwise uniqueness and continuous dependence of {SDE}s with
  non-regular drift},
        date={2011},
     journal={Stochastics},
      volume={83},
       pages={241\ndash 257},
}

\bib{huysmans2023nonuniqueness}{article}{
      author={Huysmans, Lucas},
      author={Titi, Edriss~S.},
       title={Non-{U}niqueness and {I}nadmissibility of the {V}anishing
  {V}iscosity {L}imit of the {P}assive {S}calar {T}ransport {E}quation},
        date={2025},
     journal={Journal de Math\'ematiques Pures et Appliqu\'ees},
      volume={198},
}

\bib{KrylovRockner05}{article}{
      author={Krylov, N.~V.},
      author={R\"{o}ckner, M.},
       title={Strong solutions of stochastic equations with singular time
  dependent drift},
        date={2005},
     journal={Probab. Theory Related Fields},
      volume={131},
      number={2},
       pages={154\ndash 196},
}

\bib{kumar24}{article}{
      author={Kumar, Anuj},
       title={Nonuniqueness of trajectories on a set of full measure for
  {S}obolev vector fields},
        date={2024},
     journal={Archive for Rational Mechanics and Analysis},
      volume={248},
      number={114},
}

\bib{MescoliniPitcho25}{article}{
      author={Mescolini, G.},
      author={Pitcho, J.},
      author={Sorella, M.},
       title={On vanishing diffusivity selection for the advection equation},
        date={2025},
     journal={Ann. Mat. Pura Appl.},
}

\bib{pitcho2021nonuniqueness}{article}{
      author={Pitcho, J.},
      author={Sorella, M.},
       title={Almost {E}verywhere {N}onuniqueness of {I}ntegral {C}urves for
  {D}ivergence-{F}ree {S}obolev {V}ector {F}ields},
        date={2023},
     journal={SIAM Journal on Mathematical Analysis},
      volume={55},
       pages={4640\ndash 4663},
}

\bib{Pitcho24Int}{article}{
      author={Pitcho, Jules},
       title={On the stochastic selection of integral curves of a rough vector
  field},
        date={2024},
     journal={arXiv:2407.02364},
}

\bib{PitchoArmk23}{article}{
      author={Pitcho, Jules},
       title={A remark on selection of solutions to the transport equation},
        date={2024},
     journal={Journal of Evolution Equations},
      volume={24},
}

\bib{Pitcho23selection}{article}{
      author={Pitcho, Jules},
       title={On the lack of selection for the transport equation over a dense
  set of vector fields},
        date={2025},
     journal={Journal of Functional Analysis},
}

\bib{Smirnov94}{article}{
      author={Smirnov, S.K.},
       title={Decomposition of solenoidal vector charges into elementary
  solenoids, and the structure of normal one-dimensional flows},
        date={1994},
     journal={St. Petersburg Math. J.},
      volume={5},
       pages={841\ndash 867},
}

\bib{Veretennikov82}{article}{
      author={Veretennikov, A.~Yu.},
       title={Criteria for the existence of a strong solution of a stochastic
  equation},
        date={1982},
     journal={Teor. Veroyatnost. i Primenen.},
      volume={27},
      number={3},
       pages={417\ndash 424},
}

\end{biblist}
\end{bibdiv}
\end{document}